\documentclass[12pt]{article}
\usepackage{amstext}
\usepackage{amssymb}
\usepackage{latexsym}
\usepackage{amsmath}
\usepackage{amscd}
\usepackage{amsthm}
\usepackage{graphicx}
\usepackage{array}
\usepackage{stmaryrd}
\usepackage[makeroom]{cancel}
\usepackage{url}
\usepackage{color}
\usepackage[colorlinks=true,linkcolor=blue,citecolor=red]{hyperref}

\makeatletter
\def\blfootnote{\xdef\@thefnmark{}\@footnotetext}
\makeatother

\newtheorem{theorem}{Theorem}[section]
\newtheorem{proposition}[theorem]{Proposition}
\newtheorem{lemma}[theorem]{Lemma}
\newtheorem{corollary}[theorem]{Corollary}

\newtheorem{question}[theorem]{Question}

\newcommand{\F}{\mathbb{F}}

\newcommand{\Z}{\mathbb{Z}}
\newcommand{\Q}{\mathbb{Q}}

\newcommand{\C}{\mathbb{C}}

\newcommand{\ord}{{\rm ord}}
\newcommand{\opt}{{\rm opt}}

\newcommand{\optg}{{\rm opt}}
\newcommand{\ml}{{\rm ml}}

\newcommand{\mlg}{{\rm ml}}

\newcommand{\gb}{{B}}

\title{B\"uchi's problem in modular arithmetic for arbitrary quadratic polynomials}

\date{}

\author{Pablo S\'aez, Xavier Vidaux and Maxim Vsemirnov}

\begin{document}

\maketitle

\begin{abstract}
Given a prime $p\ge5$ and an integer $s\ge1$, we show that there exists an integer $M$ such that for \emph{any} quadratic polynomial $f$ with coefficients in the ring of integers modulo $p^s$, such that $f$ is not a square, if a sequence $(f(1),\dots,f(N))$ is a sequence of squares, then $N$ is at most $M$. We also provide some explicit formulas for the optimal $M$.\blfootnote{The three authors have been partially supported by the first author Fondecyt research projects 1130134 and 1170315, Chile. The third author is partially supported by the Government of the Russian Federation (grant 14.Z50.31.0030).}
\end{abstract}

\vspace{-5pt}

{\footnotesize \noindent \emph{Keywords:} B\"uchi sequence, Hensley sequence}


\section{Introduction}

We are interested in the following question: 

\begin{question}\label{quest}
Given an integer $m\ge3$ and a quadratic polynomial $f(x)=f_2x^2+f_1x+f_0$ over $\Z$, consider the sequence 
$$
\gb_{f}^N=(f(1),\dots,f(N))
$$ 
modulo $m$. How long can this sequence be if every element of it is a square modulo $m$ but $f$ itself is not a square modulo $m$? 
\end{question}

The same question can be considered for any commutative ring $R$ with unit instead of a quotient of $\Z$. For $R=\Z$, it was first asked by R. B\"uchi in the early seventies, and was motivated by a decision problem in Logic --- see \cite{Lip90} and \cite{Maz94}. 

In the case of modular arithmetic, Question \ref{quest} was first addressed by D. Hensley in \cite[Thm. 3.1]{Hensley}. He showed that in the particular case where $m$ is an odd prime number and $f(x)$ is of the form $(x-\nu)^2-a$, then $N$ is strictly less than $m$ --- nevertheless he does not give any explicit formula for the largest possible $N$ as $\nu$ and $a$ vary. In \cite{SaezVidauxVsemirnov15}, we deal with the case where, for an odd given $m$, $f_2$ is invertible modulo $m$ --- see Theorem \ref{thJNT} below for the case of prime powers, and \cite[Section 5]{SaezVidauxVsemirnov15} for general $m$. In the present paper, we solve the problem for any given prime power and any $f$, by reducing it to the case where $f_2$ is invertible --- see Theorem \ref{small} below. 

To give a taste of our main result without introducing too many technicalities, here we state a corollary. 

\begin{theorem}\label{corsmall}
Let $p$ be a prime $\ge3$. Assume that $g_2$ is a non-zero square modulo $p$, $s$ and $t_2<s$ are positive even integers, and $f_2=p^{t_2}g_2$. As $f$ varies within the set of non-square quadratic polynomials with this restriction on $f_2$, the largest possible $N$ such that each of $f(1)$, \dots, $f(N)$ is a square modulo $p^s$, is
$$
p^{\frac{s-t_2}{2}}-1,
$$
(i.e.: there are sequences of this length, and no larger ones). 
\end{theorem}

Let us give a concrete example. For any odd prime $p$, modulo $p^4$, the polynomial $f(x)=p^2x^2+p^3$ is not the square of a polynomial, because $p^3$ is not a square, but by Hensel's lemma, it is easy to see that $f(k)$ is a square modulo $p^4$ for $k=1,\dots,p-1$, so the length $p^{\frac{4-2}{2}}-1=p-1$ is reached. Note that $f(0)$ and $f(p)$ are not squares modulo $p^4$. 

Though many analogous results exist in the literature over different type of rings $R$, they always assume that the polynomial $f$ is monic, with just one exception: Natalia Garcia-Fritz \cite[Thm. 1.6, Cor. 1.7 and the comments that follow]{GarciaFritz17} does not put restrictions on $f_2$ (unconditionally if $R$ is a function field of a curve over $\C$, and assuming the Bombieri-Lang conjecture when $R=\Q$). There is some literature on sequences of squares whose second difference is an arbitrary element of $R$, which corresponds essentially to considering a quadratic $f$ with an arbitrary dominant coefficient. Symmetric sequences of that kind were considered by Allison \cite{Allison86}, Bremner \cite{Bremner03}, Browkin and Brzezi\'nski \cite{BrowkinBrzezinski06}, and Gonzalez-Jimenez and Xarles \cite{GonzalezJimenezXarles11}. 

Analogues of Question \ref{quest} have been considered for most classical rings (but in the case of number fields, under some well-known conjectures, like Bombieri-Lang for surfaces, or some version of ABC). Relevant results in positive characteristic can be found in \cite{Pasten11} (the analogue of B\"uchi's problem for any power over fields with a prime number of elements), and in \cite{PastenWang15} (over rings of functions), who generalize previous results in \cite{PheidasVidaux06,PheidasVidaux10,ShlapentokhVidaux10,AnWang11,AnHuangWang13}. For a general survey on B\"uchi's problem and its extensions to other structures and higher powers, see \cite{PastenPheidasVidaux10}.

\section{Preliminaries and Main result}

If $n$ is an integer, $[n]_{m}$ will denote its residue class modulo $m$ (we may use the brackets notation for polynomials and for sequences as well), and if $p$ is a prime, $\ord_pn$ will stand for the usual order at $p$ of $n$, with the convention $\ord_p(0)=\infty$, so that for every integer $x$ we have, $\ord_px<\infty$ if and only if $x\ne0$. 

From now on, we will only consider sequences $\gb_{f}^N$ over $\Z$ which satisfy the two following conditions for some odd integer $m\ge3$:
\begin{enumerate}
\item[(C1)] $f(1)$, \dots, $f(N)$ are squares modulo $m$. 
\item[(C2)] $f$ is not the square of a polynomial modulo $m$. 
\end{enumerate}

 Following \cite{SaezVidauxVsemirnov15},  sequences $\gb_{f}^N$ satisfying (C1) are called \emph{$f$-B\"uchi sequences modulo $m$}, and they are called \emph{non-trivial} if (C2) is also satisfied (we may say just ``B\"uchi'' instead of ``$f$-B\"uchi'' when it is clear what $f$ is). We should immediately point out that in \cite{SaezVidauxVsemirnov15} we consider $f$-B\"uchi sequences as trivial when $f$ is the square of a polynomial of degree at most $1$. But indeed, when $f_2$ is invertible this makes no difference with condition (C2), as shown by the following proposition, which will be proved at the beginning of the next section. 

\begin{proposition}\label{PropTrivial}
Let $p$ be an odd prime and $s$ be a positive integer. Let $f=f_2X^2+f_1X+f_0\in\Z[X]$. The following statements are equivalent: 
\begin{enumerate}
\item\label{PropTrivial1} The polynomial $f$ is the square of a polynomial modulo $p^s$. 
\item\label{PropTrivial2} Either $\ord_pf_0<\min\{\ord_pf_1,\ord_pf_2\}$ and $[f_0]_{p^s}$ is a square, or $f$ is the square of a polynomial modulo $p^s$ whose degree is at most one. 
\end{enumerate}
\end{proposition}

For odd $m\ge3$, let us write $\mlg(m,f_2,f_1)$ for
$$
\max_{f_0} \left\{ N \colon \left[\gb_{f}^N\right]_m \textrm{ is a non-trivial B\"uchi sequence, where } f=f_2X^2+f_1X+f_0\right\},
$$
(with the convention $\mlg(m,f_2,f_1)=0$ if all the sequences in the set are trivial), and 
$$
\optg(m,f_2,f_1)=\mlg(m,f_2,f_1)+1.
$$
Also, we will write 
$$
\mlg(m,f_2)=\max_{f_1}\mlg(m,f_2,f_1)\qquad\textrm{and}\qquad
\optg(m,f_2)=\mlg(m,f_2)+1.
$$
Note that when $m=p$ is prime, we trivially have
$$
\mlg(p,0)=\max_{f_1}\mlg(p,0,f_1)=\max\{\mlg(p,0,f_1)\colon [f_1]_p\ne[0]_p\}.
$$
Here ``$\ml$'' stands for ``maximal length'' and ``$\opt$'' stands for ``optimal bound''. The reason to use both concepts is that the proofs are done in terms of maximal lengths but the formulas that we need from \cite{SaezVidauxVsemirnov15} are nicer in terms of the optimal bound (the reader will see the point in Theorem \ref{thJNT} below). 

We can now state our main theorem.

\begin{theorem}\label{small}
Let $p$ be a prime $\ge3$ and $s$ be a positive integer. Let $f_2\in\Z$, with $f_2\notin p^s\Z$ unless $f_2=0$. Write $t_2=\ord_pf_2$ and let $g_2$ be such that $f_2=p^{t_2}g_2$ when $f_2\ne0$, and $g_2=0$ otherwise. Assume $t_2\ne0$. We have  
$$
\optg(p^s,f_2)=
\begin{cases}
\optg(p,0)&\textrm{ if $t_2$ is odd or $t_2=\infty$,}\\
\max\{\optg(p,0),\optg(p^{s-t_2},g_2)\}&\textrm{ if $t_2$ is even.}
\end{cases}
$$
\end{theorem}

In this paper, we deal only with prime power modulus. The case of a general modulus $m$ can be reduced to the case of powers of primes following the strategy described in \cite[Section 5.1]{SaezVidauxVsemirnov15}. Any B\"uchi sequence modulo $m=p_1^{s_1}\dots p_k^{s_k}$ can be glued from B\"uchi sequences modulo the $p_i^{s_i}$ using the Chinese Remainder Theorem. The only subtle point to take care of is that one has to check that the resulting sequence modulo $m$ is non trivial, so there are various cases to consider, which result in an elementary but cumbersome analysis --- we leave the details to  the reader. 

While Theorem \ref{small} is a natural extension of what we did in our previous work, it was not clear from the beginning what the right statement should be (for example, we were surprised when we discovered that the concept of \emph{triviality} had to be kept unchanged). We have tried to write the proof in the most uniform possible way, instead of doing the obvious case by case analysis. 

In order to have a global picture of the situation, and for later references, we resume in a single theorem what we knew in the case where $f_2$ is invertible. In \cite{SaezVidauxVsemirnov15}, given $\alpha\in\Z$ not divisible by $p$, we had defined $\ml(m,\alpha)$ as
$$
\max_{a,\nu} \left\{ N \colon \left[\gb_{f}^N\right]_m\textrm{ is an $f$-B\"uchi sequence, where } f=\alpha(X+\nu)^2+a \textrm{ and }a\ne0\right\},
$$
and $\opt(m,\alpha)=\ml(m,\alpha)+1$. Note that when $f_2$ is invertible modulo an odd $m\ge3$, then every polynomial $f(X)=f_2X^2+f_1X+f_0$ can be written in a unique way in the form $\alpha(X+\nu)^2+a$ modulo $m$, so the notation in \cite{SaezVidauxVsemirnov15} is compatible with the present one. 

\begin{theorem}\emph{(\cite[Thm. 1.7, Thm. 1.8, Lem. 2.13]{SaezVidauxVsemirnov15})}\label{thJNT}
Let $p$ be an odd prime number, and $s\ge 1$ and $f_2$ be integers. Assume that $[f_2]_{p^s}$ is invertible. 
\begin{enumerate}
\item\label{thJNT1} If $[f_2]_{p^s}$ is a non-square and $p\ge5$, then $\opt(p^s,f_2)=\opt(p,n)<\infty$, where $n$ is any quadratic non-residue modulo $p$. 
\item\label{thJNT2} If $[f_2]_{p^s}$ is a non-zero square and $s=2r$ is even, then $\opt(p^s,f_2)=p^r$. 
\item\label{thJNT3} If $[f_2]_{p^s}$ is a non-zero square and $s=2r+1$ is odd, then $\opt(p^s,f_2)=\opt(p,1)p^r<\infty$. 
\item\label{thJNT4} We have $\opt(p,0)\le\frac{p+3}{2}$.
\item\label{thJNT5} For any $k\in\Z$, we have $\opt(3,2+3k)=\infty$.
\item\label{thJNT6} For any $s\ge 2$ and $k\in\Z$, we have $\opt(3^s,2+3k)=5$. 
\end{enumerate}
\end{theorem}

We get Theorem \ref{corsmall} by first applying Theorem \ref{small} and then Theorem \ref{thJNT}, items \ref{thJNT2} and \ref{thJNT4}. For other cases, it is clear how similar corollaries can be obtained.

\section{Reduction to the case when $[f_2]_{p^s}$ is invertible or is $[0]_{p^s}$}

We will frequently use the following well known fact.

\begin{lemma}\label{cor0BSx0Inv}
Let $p$ be an odd prime number. If $y\in\Z$ is a non-zero square modulo $p^t$ for some $t\geq1$, then $y$ is a square modulo $p^s$ for any $s\geq 1$.
\end{lemma}

Next, we prove Proposition \ref{PropTrivial}. 

\begin{proof}[Proof of Proposition \ref{PropTrivial}]
We first prove that \ref{PropTrivial2} implies \ref{PropTrivial1}. Assume 
$$
\ord_pf_0<\min\{\ord_pf_1,\ord_pf_2\}
$$ 
and $[f_0]_{p^s}$ is a square. If $f$ is identically $0$ modulo $p^s$, then the claim is trivial, so we may assume that $f_0$ is not $0$ modulo $p^s$. We have 
$$
f\equiv(p^rg_0)^2+p^{2r+1}Xg\equiv (p^rg_0)^2(1+pXh)\pmod{p^s}
$$ 
for some $g\in\Z[X]$, $g_0\in\Z$ not divisible by $p$, and $h\in\Z[X]$ such that $g_0^2h\equiv g\pmod{p^s}$. The Taylor series modulo $p^s$ of the square root of $1+pXh$ is actually a polynomial, since denominators are powers of $2$ and numerators have increasing order at $p$. 

We now prove that \ref{PropTrivial1} implies \ref{PropTrivial2}. Assume $s\ge2$ (indeed, for $s=1$, the claim is trivial as $\F_p$ is an integral domain). Let $\varphi\in\Z[X]$ be such that $[f]_{p^s}=[\varphi^2]_{p^s}$. We can assume $[\varphi]_{p^s}\ne[0]_{p^s}$. Let $u$ be the largest integer such that $\varphi=p^ug$ and $g\in\Z[X]$, so that $[g]_{p}\ne[0]_{p}$. We have $f\equiv\varphi^2\equiv p^{2u}g^2\pmod{p^s}$. If $2u\ge s$, there is nothing to prove, so we can assume 
$$
2u<s,
$$
hence $2u+1\le s$. Let $\tilde f=\tilde f_0+\tilde f_1X+\tilde f_2X^2\in\Z[X]$ be such that $p^{2u}\tilde f=f$ and $[g^2]_p=[\tilde f]_p$. So $[g^2]_p$ has degree at most $2$, and since $[g]_p\ne[0]_p$, we deduce that $[g]_p$ has degree at most $1$ (because we are now over the integral domain $\F_p$). So we have
$$
g=g_0+g_1X+p^v hX^2,
$$
for some $v\ge1$ and some $h\in\Z[X]$. If $h$ is the zero polynomial, then we are done. Otherwise choose $v$ as large as possible, so that $h$ has at least one coefficient not divisible by $p$, namely, $[h]_p\ne[0]_p$. We have then
$$
\varphi=p^u(g_0+g_1X+p^v hX^2). 
$$

\emph{Case 1:} Assume that $p$ divides $g_1$, so that $p$ does not divide $g_0$. We have then $\varphi=p^ug_0+p^{u+1}k_0$ for some $k_0\in\Z[X]$, hence $f\equiv\varphi^2\equiv p^{2u}g_0^2+p^{2u+1}k_1 \pmod{p^s}$ for some $k_1\in\Z[X]$, so 
$$
2u=\ord_pf_0<\min\{\ord_pf_1,\ord_pf_2\}.
$$ 

\emph{Case 2:} Assume that $p$ does not divide $g_1$. Write $\ell=g_0+g_1X$, so that
$$
f\equiv p^{2u}(\ell^2+p^v hX^2(2\ell+p^v hX^2))\pmod{p^s},
$$
hence 
\begin{equation}\label{eqow}
f-p^{2u}\ell^2\equiv p^{2u+v}hX^2(2\ell+p^vhX^2)\pmod{p^s}.
\end{equation}
If $2u+v\ge s$, then we are done (since $\ell$ has degree $1$), so it remains to consider the case where 
$$
2u+v<s,
$$ 
which will turn out to be impossible. Multiplying both sides of Equation \eqref{eqow} by 
$$
(2\ell)^{s-1}-(2\ell)^{s-2}p^vhX^2+\dots+(-1)^{s-1}(p^vhX^2)^{s-1}
$$
we obtain: 
$$
\begin{aligned}
(f-p^{2u}\ell^2)[(2\ell)^{s-1}-(2\ell)^{s-2}p^vhX^2+\dots]&\equiv
p^{2u+v}hX^2((2\ell)^s+(-1)^{s-1}(p^vhX^2)^s)\\
&\equiv p^{2u+v}hX^2(2\ell)^s \pmod{p^{s}},
\end{aligned}
$$
since $v\ge1$. Hence we have
\begin{equation}\label{eqow2}
\left(\tilde f-\ell^2\right)[(2\ell)^{s-1}-(2\ell)^{s-2}p^vhX^2+\cdots]
\equiv p^{v}hX^2(2\ell)^s \pmod{p^{v+1}},
\end{equation}
since $s-2u\ge v+1$. We now compare the coefficients of $X^{2+d+s}$ on both sides, where $d$ is the degree of $[h]_p$. Let $h_d\in\Z$ be the coefficient of $h$ at $X^d$ (so $[h_d]_p$ is the dominant coefficient of $[h]_p$), and let $h_0$ be the constant term of $h$. The coefficient of $X^{2+d+s}$ modulo $p^{v+1}$ on the left hand side is the coefficient of 
$$
\left(\tilde f-\ell^2\right)(2\ell)^{s-2}p^vhX^2
$$
modulo $p^{v+1}$ (the terms that are not written in Equation \eqref{eqow2} will have order at least $2v\ge v+1$, and in the term $(\tilde f-\ell^2)(2\ell)^{s-1}$ all monomials have degree at most $s+1<2+d+s$), which is
$$
2p^vh_0g_0\cdot2^{s-2}g_1^{s-2}p^vh_d,
$$
(here the term $2p^vh_0g_0$ comes from Equation \eqref{eqow}). On the other hand, the right hand side of Equation \eqref{eqow2} gives 
$$
p^vh_d2^sg_1^s,
$$
so we have
$$
p^v\cdot2h_0g_0\cdot 2^{s-2}g_1^{s-2}h_d\equiv
h_d2^sg_1^s\pmod{p},
$$
which is a contradiction since $p$ divides neither $g_1$ nor $h_d$. 
\end{proof}

Note that if any of $f_1$ or $f_2$ is invertible modulo $p$, then the condition 
$$
\min\{\ord_pf_1,\ord_pf_2\}>\ord_pf_0
$$ 
is never satisfied. We now prove a sequence of lemmas that will imply our main theorem. 

\begin{lemma}\label{lemRedInvf0NS}
Let $p$ be an odd prime and $s\ge1$. Let $f(X)=f_2X^2+f_1X+f_0\in\Z[X]$. Write $t_i=\ord_p f_i$ for each $i$. Assume that $\min\{t_1,t_2\}>t_0$ and $[f_0]_{p^s}$ is a non-square. If $\gb_{f}^N$ is a B\"uchi sequence modulo $p^s$, then $N=0$. 
\end{lemma}
\begin{proof}
Note that $t_0\ne\infty$. Also note that $t_0<s$ (because $[f_0]_{p^s}$ is a non-square, so in particular it is not $[0]_{p^s}$). Write $f(X)=p^{t_0}g(X)$, where $g(X)=p^{t_2-t_0}g_2X^2+p^{t_1-t_0}g_1X+g_0$, so that $g_0$ is invertible modulo $p$. We assume $N\ge1$ and will get a contradiction. Let $k,x\in\Z$ be such that $f(1)=x^2+kp^s$. From the hypothesis of the lemma, we have
$$
x^2+kp^s=f(1)=p^{t_0}g(1)\equiv p^{t_0}g_0=f_0 \pmod{p^{t_0+1}},
$$
hence, recalling that $t_0+1\le s$, $f_0$ is a square modulo $p^{t_0+1}$. Since $g_0$ is non-zero modulo $p$, also $f_0$ is non-zero modulo $p^{t_0+1}$. So $f_0$ is a square modulo $p^s$ by Lemma \ref{cor0BSx0Inv}, which contradicts our hypothesis on $f_0$. 
\end{proof}

\begin{lemma}\label{lemRedInvmint1}
Let $p$ be an odd prime and $s\ge1$. Let $f_1,f_2\in\Z$. If $f_1$ is invertible modulo $p$ and $f_2$ is not invertible modulo $p$, then
$$
\optg(p^s,f_2,f_1)=\optg(p,0,f_1)
$$
\end{lemma}
\begin{proof}
We first prove the ``$\le$'' inequality. Let $N\ge0$ and $f_0$ be integers. Write $f=f_2X^2+f_1X+f_0$ and assume that $\gb_{f}^N$ is a non-trivial B\"uchi sequence modulo $p^s$. Modulo $p$, since $f_2$ is not invertible, we have $f\equiv f_1X+f_0$. Write $g=f_1X+f_0$. Since $f(x)$ is a square modulo $p^s$ for each $x=1,\dots,N$, it is a square modulo $p$, so $\gb_g^N$ is a B\"uchi sequence modulo $p$. Since $f_1$ is invertible, $\gb_g^N$ is a non-trivial B\"uchi sequence modulo $p$, hence $N$ is at most $\mlg(p,0,f_1)$.

We now prove the other inequality. Let $h=f_1X+b$ be such that $\gb_{h}^N$ is a B\"uchi sequence of length $N=\mlg(p,0,f_1)$ (note that this is always finite by item \ref{thJNT4} of Theorem \ref{thJNT}). Consider 
$$
f=f_2X^2+f_1X+b\equiv f_1X+b\pmod p.
$$
If in the sequence $\gb_{h}^N$ there is no $0 \pmod p$, then $f(x)$ is a non-zero square modulo $p$ for any $x\in\{1,\dots,N\}$, hence it is a square modulo $p^s$ by Lemma \ref{cor0BSx0Inv}. Assume that there is some $x_0\in\{1,\dots,N\}$ such that $h(x_0)$ is congruent to $0$ modulo $p$, so that $b$ is congruent to $-f_1x_0$ modulo $p$ (there can be at most one such $x_0$). In that case, consider instead
$$
f=f_2X^2+f_1X-f_1x_0-f_2x_0^2\equiv f_1X+b\pmod p,
$$
so that $f(x_0)$ is actually $0\in\Z$, hence a square modulo $p^s$, and as before, when $x\ne x_0$, $f(x)$ is a non-zero square modulo $p^s$. In both cases, $\gb_{f}^N$ is a B\"uchi sequence modulo $p^s$. 

We now prove that $\gb_{f}$ is a non-trivial B\"uchi sequence modulo $p^s$. It is enough to prove that $f(N+1)$ is not a square modulo $p^s$. Indeed, we have $f(N+1)\equiv h(N+1)\pmod p$, and the latter is not a square by definition of $h$, so $f(N+1)$ is not even a square modulo $p$. 
\end{proof}

Next comes the key lemma for having a uniform proof of Theorem \ref{small}.

\begin{lemma}\label{lemRedInvMin}
Let $p$ be an odd prime and $s\ge1$. Let $f_1,f_2\in\Z$, with $f_i\notin p^s\Z$ unless $f_i=0$. Assume that not both $f_1$ and $f_2$ are $0$. Write $t_1=\ord_p f_1$ and $t_2=\ord_pf_2$. For $i\in\{1,2\}$, let $g_i$ be such that $f_i=p^{t_i}g_i$ (if $f_i=0$, take $g_i=0$). Write $m=\min\{t_1,t_2\}$.
\begin{enumerate}
\item\label{lemRedInvMineq} If $m$ is even, then we have:
$$
\optg(p^s,f_2,f_1)=\optg(p^{s-m},p^{t_2-m}g_2,p^{t_1-m}g_1)
$$
(where $p^{t_i-m}g_i$ reads as $0$ if $f_i=0$).
\item\label{lemRedInvMineqinf} The sides of the equation in item \ref{lemRedInvMineq} are infinite if and only if $p=3$, $s=m+1$ and $g_2\in 2+3\Z$.
\item\label{lemRedInvMineqodd} If $m$ is odd, then we have:
$$
\optg(p^s,f_2,f_1)\le3.
$$
\item\label{lemRedInvMineqother} If $t_2=\infty$ and $t_1$ is odd, then $\optg(p^s,0,f_1)\le2$. 
\end{enumerate}
\end{lemma}
\begin{proof}
We first prove items \ref{lemRedInvMineqodd} and \ref{lemRedInvMineqother}, together with the ``$\le$'' inequality in item \ref{lemRedInvMineq}. Let $N\ge0$ and $f_0$ be integers, and write $t_0=\ord_pf_0$ and $f_0=p^{t_0}g_0$ (with $g_0=0$ if $f_0=0$). Assume that $\gb_{f}^N$ is a non-trivial B\"uchi sequence modulo $p^s$, where $f=f_2X^2+f_1X+f_0$. In particular, by Proposition \ref{PropTrivial} the polynomial $f$ is not the square of a linear polynomial modulo $p^s$, and we have $m\le t_0$, unless $[f_0]_{p^s}$ is a non-square. If $m>t_0$ and $[f_0]_{p^s}$ is a non-square, we have $N=0$ by Lemma \ref{lemRedInvf0NS}, so we may assume $m\le t_0$. Write $f=p^{m}g$, where 
$$
g=p^{t_2-m}g_2X^2+p^{t_1-m}g_1X+p^{t_0-m}g_0.
$$ 

We can now complete the proof of item \ref{lemRedInvMineqodd}. Assume $m$ is odd. In that case, if $[f(n)]_{p^s}$ is a square, then $[g(n)]_{p}=[0]_p$. If $m=t_1<t_2$, we have
$$
g=p^{t_2-t_1}g_2X^2+g_1X+p^{t_0-t_1}g_0\equiv g_1X+p^{t_0-t_1}g_0\pmod p,
$$ 
hence $g(n)$ can be $0$ modulo $p$ for at most one value of $n$, hence $N\le1$. If $m=t_2\le t_1$, we have 
$$
g=g_2X^2+p^{t_1-t_2}g_1X+p^{t_0-t_2}g_0,
$$
so $g(n)$ can be $0$ modulo $p$ for at most two values of $n$, hence $N\le2$. 

We now turn to item \ref{lemRedInvMineq}. Assume $m$ is even. Since $f$ is not the square of a linear polynomial modulo $p^s$, also $g$ is not the square of a linear polynomial modulo $p^{s-m}$. Moreover, since $t_0\ge m=\min\{t_2,t_1\}$, we have 
$$
t_0-m\ge\min\{t_2-m,t_1-m\},
$$
hence $\gb_{g}^N$ is a non-trivial B\"uchi sequence modulo $p^{s-m}$ by Proposition \ref{PropTrivial}, so we have 
$$N\le \mlg(p^{s-m},p^{t_2-m}g_2,p^{t_1-m}g_1).
$$

We now prove ``$\ge$'' in item \ref{lemRedInvMineq} (so in particular, we assume that $m$ is even). First note that the claim is trivial when $t_2=0$ (which is the case in particular when $s=1$). Let 
$$
g=p^{t_2-m}g_2X^2+p^{t_1-m}g_1X+b
$$ 
be such that $\gb_{g}^N$ is a non-trivial B\"uchi sequence of length $N=\mlg(p^{s-m},p^{t_2-m}g_2,p^{t_1-m}g_1)$ ($N$ may be $\infty$). 

Consider 
$$
f=p^{m}g=p^{m}(p^{t_2-m}g_2X^2+p^{t_1-m}g_1X+b). 
$$
For any $x\in\{1,\dots,N\}$, since $g(x)$ is a square modulo $p^{s-m}$ and $m$ is even, also $f(x)$ is a square modulo $p^{s}$, so $\gb_{f}$ is a B\"uchi sequence modulo $p^s$. 

We now prove that $\gb_{f}^N$ is a non-trivial B\"uchi sequence modulo $p^s$. Since $\gb_{g}^N$ is a non-trivial B\"uchi sequence modulo $p^{s-m}$, $g$ is not the square of a linear polynomial modulo $p^{s-m}$, hence also, since $m$ is even, $p^mg$ is not the square of a linear polynomial modulo $p^s$. Moreover, by Proposition \ref{PropTrivial}, either $\min\{t_1-m,t_2-m\}\le\ord_pb$, in which case $\min\{t_1,t_2\}\le\ord_pp^mb$, or $[b]_{p^{s-m}}$ is not a square, in which case $[p^mb]_{p^s}$ is not a square. So $\gb_{f}^N$ is a non-trivial B\"uchi sequence modulo $p^s$.

We prove item \ref{lemRedInvMineqinf}. If $m=t_1<t_2$, then the right-hand side is finite by Lemma \ref{lemRedInvmint1}. 
Otherwise it is an immediate consequence of Theorem \ref{thJNT} applied to modulus $p^{s-m}$ (observe that the only case where $\optg$ is infinite is in item \ref{thJNT5}). 
\end{proof}

\begin{corollary}\label{corRedInvcoro}
Let $p$ be an odd prime and $s\ge1$. Let $f_1,f_2\in\Z$, with $f_i\notin p^s\Z$, unless $f_i=0$. Write $t_i=\ord_p f_i$ and let $g_i$ be such that $f_i=p^{t_i}g_i$ (if $f_i=0$, take $g_i=0$). We have
$$
\optg(p^s,f_2,f_1)=
\begin{cases}
\optg(p^{s-t_2},g_2,p^{t_1-t_2}g_1)&\textrm{ if $t_2\le t_1$ and $t_2$ is even,}\\
\optg(p,0,g_1)&\textrm{ if $t_2>t_1$ and $t_1$ is even.}\\
\end{cases}
$$
\end{corollary}
\begin{proof}
First note that the claim is trivial when $t_2=0$. If $t_2\le t_1$ and $t_2$ is even, this is just Lemma \ref{lemRedInvMin}. Assume that $t_2>t_1$ and $t_1$ is even. In particular, since $t_2>t_1$, $f_1$ cannot be $0$. We have 
$$
\optg(p^s,f_2,f_1)=\optg(p^{s-t_1},p^{t_2-t_1}g_2,g_1)=\optg(p,0,g_1)
$$
(recalling the convention that $p^{t_2-t_1}g_2=0$ if $f_2=0$), where the first equality comes from Lemma \ref{lemRedInvMin}, and the second equality comes from Lemma \ref{lemRedInvmint1} (which can be applied because $p^{t_2-t_1}g_2$ is not invertible modulo $p$, but $g_1$ is invertible modulo $p$ since $f_1\ne0$).
\end{proof}

\begin{lemma}\label{revival}
Let $p$ be a prime $\ge3$ and $s\ge1$. Let $f_2\in\Z$, with $f_2\notin p^s\Z$ unless $f_2=0$. Write $t_2=\ord_pf_2\ne0$. We have:
\begin{enumerate}
\item If $t_2=\infty$, then 
$$
\max\{\optg(p^s,f_2,f_1)\colon \textrm{ $\ord_p f_1<\infty$ is even}\}\ge2,
$$
\item If $t_2<\infty$ is odd, then
$$
\max\{\optg(p^s,f_2,f_1)\colon \textrm{ $\ord_p f_1<t_2$ and $\ord_p f_1$ is even}\}\ge3.
$$
\item If $t_2<\infty$ is even, then
$$
\max\{\optg(p^s,f_2,f_1)\colon \textrm{ $\ord_p f_1\ge t_2$, or $\ord_p f_1<t_2$ and $\ord_p f_1$ is even}\}\ge3.
$$
\end{enumerate}
\end{lemma}
\begin{proof}
For the first item, just note that for any non-zero $b\in\Z$ which is coprime with $p$, the function $f=b^2X$ defines a non-trivial B\"uchi sequence of length $\ge1$. Indeed, if $f\equiv (g_1X+g_2)^2\pmod{p^s}$, then $g_1^2\equiv g_ 2^2\equiv 0\pmod{p^s}$, hence $[g_1]_p=[g_2]_p=0$, but $2g_1g_2\equiv b^2\pmod{p^s}$, which contradicts the fact that $b$ is coprime with $p$, and the constant term is $0$, hence has order $\ge$ than the order of the other coefficients. We conclude by Proposition \ref{PropTrivial} that it is a non-trivial sequence. 
 
For items 2 and 3, choose $f=f_2X^2+f_1X+f_0$ with $f_1=1-3f_2$ and $f_0=2f_ 2-1$, so that $f(1)=0$ and $f(2)=1$ (so they are squares modulo any $p^s$). Since $t_2\ne0$, $f_2$ is divisible by $p$, hence $\ord_p(f_1)=0$ is even and $<t_2$. Moreover, $\gb_f^2$ is a non-trivial sequence because, one the one hand we have $\ord_p(f_0)=0\ge\min\{\ord_pf_1,\ord_pf_2\}$, and on the other hand it is not the square of a linear polynomial modulo $p^s$. If it were, then we would have
$$
f=f_2X^2+(1-3f_ 2)X+2f_2-1\equiv g_1^2X^2+2g_1g_2X+g_2^2\pmod{p^s},
$$
which is impossible, since modulo $p$ the right-hand side is a constant polynomial (because $p$ divides $f_2$, hence also $g_1$), while the left-hand side is a non-constant polynomial since $1-3f_ 2$ is not divisible by $p$. 
\end{proof}

We conclude this work with the proof of our main theorem. 

\begin{proof}[Proof of Theorem \ref{small}]
If $f_2=0$, then we have (recalling the convention $\optg(p^s,0,0)=1$ ---  see the introduction)
$$
\begin{aligned}
\optg(p^s,f_2)&=\max\{\optg(p^s,0,f_1)\colon \textrm{ $f_1\in\Z$}\}\\
&=\max\left(\{\optg(p^s,0,f_1)\colon \textrm{ $\ord_p f_1<\infty$ is even}\}\cup
\{\optg(p^s,0,f_1)\colon \textrm{ $\ord_p f_1<\infty$ is odd}\}\right)\\
&=\max\left(\{\optg(p^s,0,f_1)\colon \textrm{ $\ord_p f_1<\infty$ is even}\}\right)\\
&=\max\{\optg(p,0,g_1)\colon \textrm{ $g_1$ is invertible modulo $p$}\}\\
&=\optg(p,0),
\end{aligned}
$$ 
where the second and third equalities come from item 1 of Lemma \ref{revival} and Lemma \ref{lemRedInvMin}, and the fourth equality comes from Corollary \ref{corRedInvcoro}.

If $t_2<\infty$ is odd, then we have (using again Lemmas \ref{revival}, Lemma \ref{lemRedInvMin} and Corollary \ref{corRedInvcoro}): 
$$
\begin{aligned}
\optg(p^s,f_2)&=\max\{\optg(p^s,f_2,f_1)\colon \textrm{ $f_1\in\Z$}\}\\
&=\max(\{\optg(p^s,f_2,f_1)\colon \textrm{ $\ord_p f_1<t_2$ and $\ord_p f_1$ is even}\}\\
&\qquad\qquad\cup\{\optg(p^s,f_2,f_1)\colon \textrm{ $\ord_p f_1\ge t_2$ or $\ord_p f_1$ is odd}\})\\
&=\max\left(\{\optg(p^s,f_2,f_1)\colon \textrm{ $\ord_p f_1<t_2$ and $\ord_p f_1$ is even}\}\right)\\
&=\max\{\optg(p,0,g_1)\colon \textrm{ $g_1$ is invertible modulo $p$}\}\\
&=\optg(p,0).
\end{aligned}
$$ 

If $0\ne t_2<\infty$ is even, then we have
$$
\begin{aligned}
\optg(p^s,f_2)&=\max\left\{\optg(p^s,f_2,f_1)\colon \textrm{ $f_1\in\Z$}\right\}\\
&=\max\{\optg(p^s,f_2,f_1)\colon \textrm{ $\ord_p f_1\ge t_2$, or $\ord_p f_1<t_2$ and $\ord_p f_1$ is even,}\\
&\qquad\qquad\qquad\qquad\qquad\qquad \textrm{or $\ord_p f_1<t_2$ and $\ord_p f_1$ is odd}\}\\
&=\max\left(\{\optg(p^{s-t_2},g_2,p^{t_1-t_2}g_1)\colon \textrm{ $[g_1]_p\ne[0]_p$}\}\cup\left\{\optg(p,0,g_1)\colon \textrm{ $[g_1]_p\ne[0]_p$}\right\}\right)\\
&=\max\left\{\optg(p^{s-t_2},g_2),\optg(p,0)\right\}.
\end{aligned}
$$ 
\end{proof}


\vspace{10pt}

\noindent Pablo S\'aez\\
Independent, Chile\\
Email: pablosaezphd@gmail.com\\

\noindent Xavier Vidaux (corresponding author)\\
Universidad de Concepci\'on\\
Facultad de Ciencias F\'isicas y Matem\'aticas\\
Departamento de Matem\'atica\\
Casilla 160 C\\
Email: xvidaux@udec.cl\\
Tel.: +56 9 61 55 28 08\\

\noindent Maxim Vsemirnov\\
St.~Petersburg Department of V.A.Steklov Institute of Mathematics, \\
27 Fontanka, St.~Petersburg, 191023, Russia\\
and \\
St.~Petersburg State University, \\
Department of Mathematics and Mechanics, \\
28 University prospekt, St. Petersburg, 198504,  Russia\\
Email: vsemir@pdmi.ras.ru


\begin{thebibliography}{99}
\bibitem[All86]{Allison86} D. Allison, \emph{On square values of quadratics}, Math. Proc. Cambridge Philos. Soc. {\bf 99}, no. 3, 381--383 (1986).

\bibitem[AHW13]{AnHuangWang13} T. T. H. An, H.-L. Huang and J. T.-Z. Wang, \emph{Generalized B\"uchi's problem for algebraic functions and meromorphic functions}, Math. Z. {\bf 273}, no. 1-2, 95--122 (2013).

\bibitem[AW11]{AnWang11} Ta Thi Hoai An and J. Tzu-Yueh Wang, \emph{Hensley's problem for complex and non-Archimedean meromorphic functions}, Journal of Mathematical Analysis and Applications, {\bf 381-2}, 661-677 (2011). 

\bibitem[Bre03]{Bremner03} A. Bremner, \emph{On square values of quadratics}, Acta Arith. {\bf 108}, 95--111 (2003).


\bibitem[BB06]{BrowkinBrzezinski06} J. Browkin and J. Brzezi\'nski, \emph{On sequences of squares with constant second differences}, Canad. Math. Bull. {\bf 49-4}, 481-491 (2006). 


\bibitem[Ga17]{GarciaFritz17} N. Garcia-Fritz, \emph{Quadratic sequences of powers and Mohanty's conjecture}, Int. J. Number Theory {\bf 14}, no. 2, 479--507 (2018). 

\bibitem[GoX11]{GonzalezJimenezXarles11} E. Gonz\'alez-Jim\'enez and X. Xarles, \emph{On symmetric square values of quadratic polynomials}, Acta Arith. {\bf 149}, no. 2, 145--159 (2011).

\bibitem[He]{Hensley} D. Hensley, \emph{Sequences of squares with second difference of two and a problem of logic}, unpublished (1980-1983).

\bibitem[Lip90]{Lip90} L. Lipshitz, \emph{Quadratic forms, the five square problem, and diophantine equations}, The collected works of J. Richard B\"uchi (S. MacLane and Dirk Siefkes, eds.) Springer, 677-680, (1990).

\bibitem[Maz94]{Maz94} B. Mazur, \emph{Questions of decidability and undecidability in number theory}, The Journal of Symbolic Logic {\bf 59-2}, 353-371 (1994).

\bibitem[Pa11]{Pasten11} H. Pasten, \emph{B\"uchi's problem in any power for finite fields}, Acta Arithmetica {\bf 149-1}, 57-63 (2011).

\bibitem[PaPhV10]{PastenPheidasVidaux10} H. Pasten, T. Pheidas, X. Vidaux, \emph{A survey on B\"uchi's problem\,: new presentations and open problems}, Proceedings of the Hausdorff Institute of Mathematics, Zapiski POMI {\bf 377}, 111-140, Steklov Institute of Mathematics. Published online http://www.pdmi.ras.ru/znsl/2010/v377.html  (2010).

\bibitem[PaW15]{PastenWang15} H. Pasten and J. T.-Y. Wang, \emph{Extensions of B\"uchi's higher powers problem to positive characteristic}, Int. Math. Res. Not. IMRN 2015, no. 11, 3263--3297. 

\bibitem[PhV06]{PheidasVidaux06} T. Pheidas and X. Vidaux, \emph{The analogue of B\"uchi's problem for rational functions}, The Journal of The London Mathematical Society {\bf 74-3}, 545-565 (2006).

\bibitem[PhV10]{PheidasVidaux10} T. Pheidas and X. Vidaux, \emph{Corrigendum\,: The analogue of B\"uchi's problem for rational functions}, The Journal of the London Mathematical Society {\bf 82-1}, 273-278 (2010).

\bibitem[SVV15]{SaezVidauxVsemirnov15} P. S\'aez, X. Vidaux and M. Vsemirnov, \emph{Optimal bounds for B\"uchi's problem in modular arithmetic}, Journal of Number Theory {\bf 149}, 368--403 (2015).

\bibitem[ShV10]{ShlapentokhVidaux10} A. Shlapentokh and X. Vidaux, \emph{The analogue of B\"uchi's problem for function fields}, Journal of Algebra {\bf 330-1}, 482-506 (2010).

\bibitem[Vo00]{Vojta00} P. Vojta, \emph{Diagonal quadratic forms and Hilbert's Tenth Problem}, Contemporary Mathematics {\bf 270}, 261-274 (2000).

\end{thebibliography}
\end{document}